\documentclass[a4paper,10pt]{amsart}

\usepackage{graphics,graphicx}
\usepackage{amssymb,amsmath}
\usepackage[all]{xy}
\usepackage{stmaryrd}
\usepackage{multirow}
\CompileMatrices
\usepackage{hyperref}

\newtheorem{theorem}{Theorem}[section]
\newtheorem{lemma}[theorem]{Lemma}
\newtheorem{proposition}[theorem]{Proposition}

\theoremstyle{definition}

\newtheorem{definition}[theorem]{Definition}
\newtheorem{example}[theorem]{Example}
\newtheorem{remark}[theorem]{Remark}

\newtheorem{algorithm}[theorem]{Algorithm}

\theoremstyle{remark}

\def\aa{\ensuremath{\mathfrak{a}}}

\def\cc{\ensuremath{\mathfrak{c}}}
\def\TT{\mathbb{T}}
\def\KK{\mathbb{K}}
\def\ZZ{\mathbb{Z}}
\def\QQ{\mathbb{Q}}
\def\OO{\mathcal O}
\def\myident{\hspace*{0.6cm}}
\def\<{\langle}
\def\>{\rangle}
\def\trop#1{\mathrm{trop}(#1)}
	
\def\good{/\!\!/}

\newcommand{\GKZ}{{\rm GKZ}}
\newcommand{\gkz}{\GKZ}
\newcommand{\dual}{\vee}
\newcommand{\seite}{\preceq}

\newcommand{\gdw}{\Leftrightarrow}

\newcommand{\cone}[1]{\mathrm{cone}(#1)}

\makeatletter
\newcommand{\thickhline}{%
    \noalign {\ifnum 0=`}\fi \hrule height 1pt
    \futurelet \reserved@a \@xhline
}
\makeatother

\begin{document}

\markboth{Simon~Keicher}
{Computing the GIT-fan}

%===============================
% title:
\author[S.~Keicher]{Simon~Keicher}
\address{Mathematisches Institut, Universit\"at T\"ubingen,
Auf der Morgenstelle 10, 72076 T\"ubingen, Germany}
\email{keicher@mail.mathematik.uni-tuebingen.de}

\title{Computing the GIT-fan}

\maketitle
%===============================

\begin{abstract}
We present an algorithm to compute the GIT-fan of 
algebraic torus actions on affine varieties.
\end{abstract}

% \keywords{GIT-fan; variation of GIT-quotients; torus actions.}

\section{Introduction}

Given an action of a connected reductive linear algebraic
group~$H$ on an algebraic variety~$X$, Mumford constructed open
$H$-sets $U \subseteq X$ admitting a good quotient
$U \to U \good H$,
see~\cite{mumford:GIT}.
His construction depends on the choice of a $H$-linearized
ample line bundle on $X$
and, in general, one obtains several distinct quotients.
This variation of GIT-quotients is described by a
combinatorial structure, the so-called \textit{GIT-fan};
see the work by Dolgachev/Hu~\cite{dolgachevhu:GIT}
and Thaddeus~\cite{thaddeus:GIT} for ample
bundles on a projective variety $X$ and~\cite{gitviacoxrings}
for the affine case.

In the present note, we provide an algorithm for computing
the GIT-fan describing the quotients arising from the possible
linearizations of the trivial bundle for the case that $X$
is affine and $H$ is a torus.
Note that the torus case is essential for the general
one: if a connected reductive group $G$ acts on $X$, then
the associated GIT-fan equals that of the action of the
torus $G/G^s$ on the affine variety $X \good G^s$, where
$G^s \subseteq G$ is a maximal connected semisimple
subgroup, see~\cite{gitviacoxrings}.
Moreover, our setting also occurs in the context of Mori dream
spaces: there the Neron-Severi torus acts on the total
coordinate space and the GIT-fan of this action is precisely
the Mori chamber decomposition of the effective cone,
see~\cite{mdsgit}.

Our algorithm is based on the construction of the GIT-fan
provided in~\cite{hausen:amplecone}.
We assume that $X \subseteq \KK^r$ is given by concrete
equations.
The main computational steps are to determine the toric
orbits of $\KK^r$ intersecting $X$,
see section \ref{section:afaces},
a suitable number of the so-called
\textit{orbit cones} of the $H$-action on $X$
and the \textit{GIT-chamber} of a given weight,
see section \ref{section:chambers}.
The GIT-fan is then obtained by traversing a spanning
tree of its dual graph;
this idea also shows up
in the computation of Gr\"obner-fans, symmetric fans
and tropical varieties as presented
in~\cite{groebfan,symmfans,jensen:tropvars}.
We discuss some examples in section \ref{section:examples}.
At the moment, a \texttt{Maple/convex}~\cite{convex}
implementation of our algorithm is
available~\cite{gitfanlib}.

The author would like to thank J\"urgen~Hausen for valuable discussions and comments and 
 the referee for helpful suggestions.

\section{Computing the GIT-fan}\label{section:chambers}

Throughout the whole note, $\KK$ is an algebraically closed 
field of characteristic zero.
In this section, we first recall the necessary 
concepts from~\cite{hausen:amplecone} 
and thereby fix our notation.
Then we present and prove our algorithms for the GIT-fan.
Aspects of efficiency of the algorithms are discussed 
at the end of this section.

We will work with the following description of the toric 
orbits of $\KK^r$ in terms of faces of the orthant 
$\gamma := \QQ_{\geq 0}^r$: 
the standard torus $\TT^r:=(\KK^*)^r$ acts via
\[
\TT^r\times \KK^r\to \KK^r\,,\qquad
 t\cdot x = (t_1x_1,\ldots,t_rx_r)\,.
\]
Given a face $\gamma_0\seite\gamma$, define the reduction 
of an $r$-tuple $z$, of e.g.~numbers, along $\gamma_0$ as
\[
  z_{\gamma_0} := (z_1',\ldots,z_r')\,,
 \qquad
  z_i':=\begin{cases}
         z_i, & e_i\in \gamma_0\\
 	 0,   & e_i\not\in \gamma_0,
        \end{cases}
\]
where $e_1,\ldots,e_r \in \QQ^r$ denote  the canonical basis 
vectors. Then, one has a bijection
\[
 \{\text{ faces of $\gamma$ } \}
\ \leftrightarrow\ 
 \{\text{ $\TT^r$-orbits }\}\,,
 \qquad
 \gamma_0\ \mapsto\ \TT^r_{\gamma_0} := \{ t_{\gamma_0};\ t\in\TT^r \}\,.
\]
Note, that in the notation of \cite{fulton}, 
$\TT^r_{\gamma_0}$ is the $\TT^r$-orbit 
through the distinguished point corresponding to the dual face 
$\gamma_0^*:=\gamma^\perp\cap\gamma^\dual\seite\gamma^\dual$.

\begin{definition}\label{def:aface}
Let $\aa\subseteq\KK[T_1,\ldots,T_r]$ be an ideal.
A face $\gamma_0$ of the positive orthant $\gamma$ is an
 \textit{$\aa$-face} if 
$V(\TT^r_{\gamma_0};\,\aa) \not=\emptyset$.
\end{definition}

If $X \subseteq \KK^r$ is the zero set of 
the ideal $\mathfrak{a} \subseteq \KK[T_1,\ldots, T_n]$,
then the $\mathfrak{a}$-faces correspond exactly to the 
$\TT^r$-orbits intersecting $X$ nontrivially.
The computation of $\aa$-faces will be discussed in 
section \ref{section:afaces}.

We are ready to introduce GIT-chambers and the GIT-fan.
Assume that the defining ideal 
$\aa\subseteq\KK[T_1,\ldots,T_r]$
of $X \subseteq \KK^r$ is monomial-free
and homogeneous with respect to a 
$\ZZ^k$-grading
\[
 q_i := \deg(T_i)\in \ZZ^k\,,\qquad
 1\leq i\leq r\,.
\]
Then the corresponding action of the torus 
$H = \TT^k$ on $\KK^r$ leaves the zero 
set $X = V(\KK^r;\,\aa)\subseteq \KK^r$
invariant.
Let $Q$ be the $k \times r$ matrix with 
columns $q_1,\ldots,q_r$.
We assume that the cone $Q(\gamma) \subseteq \QQ^k$ 
is of dimension $k$.

A \textit{projected $\aa$-face} is a cone $Q(\gamma_0)$ with 
$\gamma_0\seite\gamma$ an $\aa$-face.
In \cite{hausen:amplecone} these are called \textit{orbit cones}.
Write $\Omega_\aa$ for the set of all projected $\aa$-faces.

\begin{definition}
The \textit{GIT-chamber} of a vector 
$w \in Q(\gamma) = \cone{q_1,\ldots,q_r} \subseteq \QQ^k$ 
is the convex, polyhedral cone
\[
\lambda(w) 
:= 
\bigcap_{w\in\vartheta\in \Omega_\aa} \vartheta
\subseteq 
\QQ^k.
\]
The \textit{GIT-fan} of the $H$-action on $X = V(\KK^r;\,\aa)$ 
is the set $\Lambda(\aa,Q) = \{\lambda(w);\,w \in Q(\gamma)\}$ 
of all GIT-chambers.
\end{definition}

As the name suggests, $\Lambda(\aa,Q)$ is indeed a fan in $\QQ^k$ 
with $Q(\gamma)$ as its support, 
see~\cite[Thm. III.1.2.8]{coxrings}.
Note, however, that the cones of the GIT fan need not be pointed
in general.
The set of $j$-dimensional cones of $\Lambda(\aa,Q)$ will be 
denoted by $\Lambda(\aa,Q)^{(j)}$.

We turn to the computation of GIT-chambers.
Let $\Omega := \{Q(\gamma_0);\,\gamma_0\seite\gamma\}$ 
be the set of projected faces of $\gamma$ and let 
$\Omega^{(j)}\subseteq\Omega$ be the subset of 
$j$-dimensional cones.
Similarly, $\Omega_\aa^{(j)} \subseteq \Omega_\aa$ 
is the subset of $j$-dimensional projected $\aa$-faces.
We have
\[
\Omega_\aa^{(k)}
\subseteq
\Omega_0^{(k)} 
:= 
\left\{\vartheta\in\Omega^{(k)};\,
\text{ all facets of $\vartheta$ are in }
\Omega_\aa^{(k-1)}\right\}
\subseteq
\Omega^{(k)}
\,,
\]
where the first containment is due to the fact 
that faces of projected $\aa$-faces are again 
projected $\aa$-faces, 
see \cite[Cor. 2.4]{hausen:amplecone}.
Given a vector $w$ in the relative interior 
$Q(\gamma)^\circ$, 
set $\Omega^{(k)}(w)$ for the collection of all
 $\vartheta\in\Omega^{(k)}$
 that contain $w$.
The next algorithm determines 
the associated GIT-chamber 
$\lambda = \lambda(w)$.

\goodbreak
\begin{remark}\label{rem:efforbitcones}
\begin{enumerate}
\item 
The set $\Omega^{(j)}$ is computed directly
by taking cones over suitable subsets of 
$\{q_1, \ldots, q_r\}$.
\item
The computation of $\Omega^{(j)}(w)$ can be sped 
up via \textit{point location}~\cite{pointloc},
i.e. we only consider cones $\vartheta\in\Omega^{(k)}$ with 
at least one generator lying on the same side as $w$ 
of a random hyperplane subdividing $Q(\gamma)$. 
\item
For an efficient computation of
$\Omega_\aa^{(j)}$, one reduces the amount
of $\aa$-face tests as follows.
Check for any $\vartheta\in\Omega^{(j)}$ if some 
$\gamma_0 \preceq \gamma$ with 
$Q(\gamma_0) = \vartheta$ is an  $\aa$-face.
As soon as such a face has been found, 
all other faces projecting to $\vartheta$
may be ignored in subsequent tests.
\end{enumerate}
\end{remark}

\begin{algorithm}[GIT-chamber]\label{algo:chamberoc}
Let $w\in Q(\gamma)^\circ$ be given.
Assume that $\Omega^{(k)}(w)$ and 
$\Omega_\aa^{(k-1)}$ are known.

\begin{enumerate}
\item[\tiny 1] $\lambda := \QQ^k$
\item[\tiny 2] {\bf for each} $\vartheta\in\Omega^{(k)}(w)$ 
\item[\tiny 3] \myident {\bf if} $\vartheta\not\supseteq \lambda$ 
{\bf and} all facets of $\vartheta$ are in $\Omega_\aa^{(k-1)}$ 
\item[\tiny 4] \myident\myident $\lambda := \lambda \cap \vartheta$
\item[\tiny 5] {\bf return} $\lambda$
\end{enumerate}
\end{algorithm} 

\begin{lemma}\label{lem:facesintersections}
 Let $\Sigma\subseteq\QQ^k$ be a pure $k$-dimensional fan  
with convex support $|\Sigma|$ 
and let 
$\tau\in\Sigma$ be such that 
$\tau\cap |\Sigma|^\circ\ne \emptyset$.
Then $\tau$ is the intersection 
over all $\sigma\in\Sigma^{(k)}$ 
satisfying $\tau\seite\sigma$.
\end{lemma}

\begin{lemma}\label{lem:detbywalls2}
Let $\lambda\in\Lambda(\aa,Q)^{(k)}$ and $\vartheta_0\in\Omega_0^{(k)}$.
If $\vartheta_0^\circ \cap \lambda^\circ \not= \emptyset$
then $\lambda\subseteq\vartheta_0$.
\end{lemma}

\begin{proof}
 Suppose 
$\lambda\not\subseteq\vartheta_0$.
Choose $w\in\lambda^\circ \setminus \vartheta_0$ and 
$v\in\vartheta_0^\circ\cap\lambda^\circ$.
Then $\cone{v,w}\cap (\vartheta_0\setminus \vartheta_0^\circ)$ 
lies on some facet  $\eta_0\seite\vartheta_0$. 
By construction, $\eta_0^\circ\cap\lambda^\circ\not=\emptyset$.
Since $\eta_0\in \Omega_\aa^{(k-1)}$ holds, 
$\lambda$ is not a GIT-chamber; a contradiction.
\end{proof}

\begin{proof}[Proof of Algorithm \ref{algo:chamberoc}]
The algorithm terminates with a cone 
$\lambda \subseteq \QQ^k$ containing the given 
$w \in Q(\gamma)^\circ$ and our task is to show that 
$\lambda = \lambda(w)$ holds.
For this we establish
\[
\lambda
\ = \ 
\bigcap_{w \in \vartheta \in \Omega_0^{(k)}} \vartheta
\ = \ 
\bigcap_{w \in \vartheta \in \Omega_{\aa}^{(k)}} \vartheta
\ = \ 
\lambda(w).
\]
The first equality is due to the algorithm.
The third one follows from Lemma~\ref{lem:facesintersections}.
Moreover, in the middle one, the inclusion
``$\subseteq$'' follows from 
$\Omega_0^{(k)} \supseteq \Omega_{\aa}^{(k)}$.
Thus we are left with verifying ``$\supseteq$''
of the middle equality.

First suppose that $\lambda(w)$ is of full 
dimension.
Then, for any $\vartheta_0 \in \Omega_0^{(k)}$
with $w \in \vartheta_0$, we obtain 
$\vartheta_0^\circ \cap \lambda(w)^\circ \ne \emptyset$,
because $w \in \lambda(w)^\circ$ holds.
Lemma~\ref{lem:detbywalls2} shows 
$\lambda(w) \subseteq \vartheta_0$. 
Thus, we obtain $\lambda \supseteq \lambda(w)$.
The case of $\dim(\lambda(w)) < k$ then follows
from the observation that $\lambda(w)$ is 
the intersection over all 
fulldimensional chambers $\lambda(w')$ with $w\in\lambda(w')$,
see Lemma~\ref{lem:facesintersections}.
\end{proof}

Working with $(k-1)$-dimensional projected 
$\aa$-faces in Algorithm~\ref{algo:chamberoc} 
simplifies the necessary $\aa$-face tests 
compared to the following naive variant of the 
algorithm using $k$-dimensional ones.

\begin{algorithm}[GIT-chamber, v2]
\label{algo:chamberocv2}
Let $w\in Q(\gamma)^\circ$ be given
and assume that $\Omega^{(k)}(w)$ 
is known.

\begin{enumerate}
\item[\tiny 1] $\lambda := \QQ^k$
\item[\tiny 2] {\bf for each} $\vartheta\in\Omega^{(k)}(w)$ 
\item[\tiny 3] \myident {\bf if} $\vartheta\not\supseteq \lambda$ 
{\bf and} there is an $\aa$-face $\gamma_0 \preceq \gamma$ 
with $Q(\gamma_0) = \vartheta$ 
\item[\tiny 4] \myident\myident $\lambda := \lambda \cap \vartheta$
\item[\tiny 5] {\bf return} $\lambda$
\end{enumerate}
\end{algorithm} 

The naive variant~\ref{algo:chamberocv2}, in contrast, 
involves fewer convex geometric operations as~\ref{algo:chamberoc}
and thus can be more efficient if the latter ones 
are limiting the computation.
See Remark~\ref{rem:tables} for a more concrete comparison 
of complexity aspects.

% ---- traversal -----------

We turn to the GIT-fan. 
Given a full-dimensional cone $\lambda\subseteq \QQ^k$, 
we denote by ${\rm innerfacets}(\lambda)$ the set 
of all facets of $\lambda$ that intersect the relative 
interior $Q(\gamma)^\circ$.
Moreover, for two sets $A,B$,
we shortly write $A \ominus B$ for $(A \cup B ) \setminus (A \cap B)$.
The following algorithm computes the set of maximal 
cones of the GIT-fan $\Lambda(\aa,Q)$.

\begin{algorithm}[GIT-fan]\label{algo:gitfantraversal} \

\begin{enumerate}
\item[\tiny 1] $\Lambda:= \{\lambda_0\}$ with a 
random full-dimensional GIT-chamber $\lambda_0$
\item[\tiny 2] $\mathcal{F}:= {\rm innerfacets}(\lambda_0)$
\item[\tiny 3] {\bf while} there is $\eta \in \mathcal{F}$
\item[\tiny 4] 
\myident Compute the full-dimensional 
GIT-chamber $\lambda'\not\in\Lambda$ with $\eta\seite\lambda'$ 
\item[\tiny 5]  
\myident $\Lambda := \Lambda \cup \{\lambda'\}$
\item[\tiny 6]  
\myident $\mathcal{F}:= \mathcal{F} \ominus {\rm innerfacets}(\lambda')$
\item[\tiny 7]  {\bf return} $\Lambda$
\end{enumerate}

\end{algorithm}

\begin{remark}
In the fourth line of the algorithm, 
let $\lambda\in\Lambda$ be the already found GIT-chamber 
with facet $\eta$.
Then $\lambda'=\lambda(w')$ can be calculated with 
Algorithm~\ref{algo:chamberoc}, where 
$ w' := w(\eta) - \varepsilon\cdot v$
for some $w(\eta)\in\eta^\circ$
and $v\in \lambda^\dual\cap \eta^\perp$
with a suitably small $\varepsilon >0$.
One possibly must reduce $\varepsilon$ until 
$\lambda(w')\cap \lambda = \eta$.
\end{remark}

\begin{proof}[Proof of Algorithm~\ref{algo:gitfantraversal}]
Write $\vert \Lambda \vert$
for the union over all $\lambda \in \Lambda$ 
and $\vert \mathcal{F} \vert$ for the 
union over all $\eta \in \mathcal{F}$.
Then, in each passage of the loop, 
a full-dimensional chamber of 
$\Lambda(\aa,Q)$ is added to $\Lambda$ 
and, after adapting,
$\vert \mathcal{F} \vert \cap Q(\gamma)^\circ$
is the boundary of 
$\vert \Lambda \vert \cap Q(\gamma)^\circ$
with respect to $Q(\gamma)^\circ$.
The set $\mathcal{F}$ is empty if and only 
if $\vert \Lambda \vert$ equals $Q(\gamma)$.
This shows that the algorithm terminates
with the collection of maximal cones of 
$\Lambda(\aa,Q)$ as output.
\end{proof}

Note that Algorithm~\ref{algo:gitfantraversal} traverses a 
spanning tree of the (implicitly known) 
\textit{dual graph} of $\Lambda(\aa,Q)$ which
has the maximal cones as its vertices
and any two are connected by an edge if they 
share a common facet.
Another traversal method for implicitly known graphs 
is \textit{reverse search} by Avis and 
Fukuda~\cite{reversesearch}, which also 
might be applied to our problem by the following 
observation.

\begin{proposition}
The GIT-fan $\Lambda(\aa,Q)$ is the normal fan of a 
polyhedron.
If $\QQ_{\geq 0}^k \subseteq Q(\gamma)$, then
$\Lambda(\aa,Q)^{(k)}$ can be enumerated using 
reverse search.
\end{proposition}

\begin{proof}
The first statement is~\cite[Cor.~10.4]{polyhedraldivisors}.
The second claim follows from the first one 
and~\cite[Sec.~3]{groebfan}.
\end{proof}

\begin{remark}
\label{rem:tables}
We compare the usage of Algorithm~\ref{algo:chamberoc} 
(in~\ref{algo:gitfantraversal})
to that of~\ref{algo:chamberocv2}.
As a test, we compute the GIT-fans of the maximal
torus action on the (affine cones over the) Grassmannians 
$G(2,5)$ and $G(2,6)$, using a \texttt{Maple/convex} implementation.
The following table lists the total number of $\aa$-face tests 
and the total number of cones $\vartheta$ entering 
the fourth line of Algorithms~\ref{algo:chamberoc} 
and~\ref{algo:chamberocv2}

\begin{center}
\footnotesize
\begin{tabular}{ccccc}
 	\thickhline
        & \multicolumn{2}{c}{Algorithm~\ref{algo:gitfantraversal} with~\ref{algo:chamberoc}} & \multicolumn{2}{c}{Algorithm~ \ref{algo:gitfantraversal} with~\ref{algo:chamberocv2}}\\
	& $\sharp$ \aa-face-tests & $\sharp$ cones $\vartheta$ & $\sharp$ \aa-face-tests & $\sharp$ cones $\vartheta$ \\ \hline
	$G(2,5)$ & $300$ & $21$ & $469$ & $20$ \\ 
	$G(2,6)$ & $6574$ & $50$ & $21012$ &  $52$ \\ \thickhline
\end{tabular}
\end{center}

\noindent
Note that in Algorithm~\ref{algo:chamberoc}, the \aa-face 
tests concern faces of lower dimension than in
 Algorithm~\ref{algo:chamberocv2} and thus are even faster.
\end{remark}

\goodbreak
\begin{remark}
\begin{enumerate}
\item 
Intermediate storage of occurring cones and their intersections in 
Algorithms~\ref{algo:chamberoc} and~\ref{algo:chamberocv2} 
saves time.
 \item  
The traversal of the GIT-fan can take advantage of symmetries as explained in~\cite[Ch. 3.1]{symmfans}.
\end{enumerate}
\end{remark}

%%%%%%%%%%%%%%%%% a-faces %%%%%%%%%%%%%

\section{Computing $\aa$-faces}
\label{section:afaces}

Let $X \subseteq \KK^r$ be the zero set of 
an ideal $\aa \subseteq \KK[T_1,\ldots, T_r]$.
Here we compute torus orbits of $\KK^r$ 
intersecting $X$. In the notation of 
section~\ref{section:chambers}, this means to 
determine the $\aa$-faces $\gamma_0 \preceq \gamma$ 
of the orthant $\gamma = \QQ_{\geq 0}^r$.

Given a face $\gamma_0\seite\gamma$ 
and a polynomial $f\in\KK[T_1,\ldots,T_r]$, we write 
$f_{\gamma_0}:=f(T_{\gamma_0})\in\KK[T_{\gamma_0}]$
where $T:=(T_1,\ldots,T_r)$, 
i.e. we replace 
each $T_i$ with zero if $e_i\not\in\gamma_0$. 
Let 
$\aa_{\gamma_0}:= \< f_{\gamma_0};\,f\in \aa\>\subseteq \KK[T_{\gamma_0}]$.
A direct $\aa$-face test is the following, 
based on a radical membership problem.

\begin{remark}
\label{rem:afacesnaive}
A  face $\gamma_0\seite\gamma$ is an $\aa$-face 
if and only if 
$\prod_{e_i\in\gamma_0}T_i\not\in\sqrt{\aa_{\gamma_0}}$. 
\end{remark}

This leads to a Gr\"obner based way to decide 
whether a given $\gamma_0 \preceq \gamma$ 
is an $\aa$-face.
The main aim of this section is speed up this 
direct approach by dividing out all possible 
torus symmetry.
This is done in A
lgorithm~\ref{algo:afacespushed}.
Further possible improvements are discussed at 
the end of the section.

First consider any torus $\TT$ and an ideal 
$\cc\subseteq \OO(\TT)$.
Let $H\subseteq \TT$ be the maximal subgroup leaving 
$V(\TT;\,\cc)$ invariant
and denote by $\pi\colon\TT \to \TT/H$ the quotient 
map.
To describe $\pi$ explicitly, we use the 
correspondence between integral matrices 
and homomorphisms of algebraic tori:
every $n \times k$ matrix $A$ 
defines a homomorphism $\alpha\colon\TT^k\to\TT^n$ 
by sending 
$t\in\TT^k$ to $(t^{A_{1*}},\ldots,t^{A_{n*}})\in\TT^n$
where the $A_{i*}$ are the rows of $A$.

\begin{remark}
\label{rem:gradiator}
The map $\pi\colon\TT^k\cong \TT \to \TT/H\cong \TT^n$ 
is given by any $n \times k$  matrix $P$ of full rank satisfying
\[
 \ker(P)\ =\ 
 \bigcap_{g\in \cc} \ker(P_g)\,,
\]
where to $g=a_0T^{\nu_0}+\ldots+a_mT^{\nu_m}\in\cc$ we assign 
the $m \times k$ matrix $P_g$ with rows 
$\nu_1 - \nu_0,\ldots,\nu_m - \nu_0$.
\end{remark}

\begin{remark}
\label{rem:gradiator:algo}
Fix a generating set $G:=(g_1,\ldots,g_l)$ of 
$\cc\subseteq\KK[T_1^{\pm 1},\ldots,T_k^{\pm 1}]$.
Let $P_{G}$ be the stack matrix, i.e. the vertical concatenation,
 of $P_{g_1},\ldots, P_{g_l}$.
Compute the Hermite normal form $D = U\cdot P_{G}$ 
with a unimodular matrix $U$.
Choose $P$ as the matrix consisting of the 
upper non-zero rows of $D$.
Then $P$ describes $\pi\colon\TT^k\cong \TT \to \TT/H\cong \TT^n$.
\end{remark}

\begin{proof}
Clearly, $P$ is of full rank.
Since the exponent vectors of each $g\in\cc$ are 
linear combinations of the exponent vectors of 
$g_1,\ldots,g_l$, we have
\[
\ker(P)
= \ker(P_{G}) 
= \bigcap_{i=1}^l \ker(P_{g_i})
= \bigcap_{g\in\cc} \ker(P_{g})\,.
\]
\end{proof}

A \textit{push forward} of $g\in\cc$ under 
$\pi$ is a $h\in \OO(\TT/H)$ satisfying 
$\pi^*h = T^\mu g$ for some monomial $T^\mu$;
we simply write $\pi_*g$ for any such $h$ and 
\[
\pi_*\cc := \<\pi_*g;\ g\in \cc\>\subseteq \OO(\TT/H)\,. 
\]

\begin{remark}
\label{algo:pushforward}
Let $P\colon\ZZ^k\to \ZZ^n$ be as in~\ref{rem:gradiator}
and let $g= a_0T^{\nu_0} + \ldots + a_mT^{\nu_m}\in \cc$.
Compute a Smith normal form $D= U\cdot P\cdot V$ with unimodular matrices $U,V$.
Define $B := P_g\cdot V\cdot \sigma\cdot U^{-1}$ where 
$\sigma$ is a rational section for $D$
 and $P_g$ is as in~\ref{rem:gradiator}.
Then there is $\mu\in\ZZ_{\geq 0}^n$ such that
\[
\pi_*g = T^\mu\left(a_0 + a_1T^{B_{1*}}+\ldots +a_mT^{B_{m*}}\right)\in \KK[T_1,\ldots,T_n]\,.
\]
\end{remark}

\begin{proof}
Let $\pi_g\colon\TT^k\to\TT^m$ and 
$\beta\colon\TT^n\to\TT^m$ be the maps of tori 
defined by the matrices $P_g$ and $B$.
Clearly, $g = T^{\kappa}\pi_g^*h$ for 
$h:=a_0 + a_1T_1 + \ldots + a_mT_m\in \OO(\TT^m)$ 
and some $\kappa\in \ZZ_{\geq 0}^k$.
Each $g\in\cc$ is $H$-homogeneous.
This implies $\ker(P)\subseteq\ker(P_g)$, 
so there is a unique integral matrix $B'$ such that 
$P_g = B'\cdot P$.
In particular, $B=B'$ is integral.
Therefore, $g=T^\kappa\pi^*(\beta^*h)$.
\end{proof}

We now specialize to the case of $\aa$-face-verification.
Given $\gamma_0\seite\gamma$, let 
$H(\gamma_0) \subseteq \TT^r_{\gamma_0}$ be the maximal 
subgroup leaving $V(\TT^r_{\gamma_0};\,\aa_{\gamma_0})$ 
invariant.
Our approach reduces the dimension of the problem by using
\[
 V(\TT^r_{\gamma_0};\,\aa)\,\not=\,\emptyset
\quad \gdw \quad
V\left(\TT^r_{\gamma_0}/H(\gamma_0);\,\pi_*\aa_{\gamma_0}\right)\,\not=\,\emptyset\,.
\]

\begin{algorithm}[\aa-face verification]\label{algo:afacespushed}
Let $\aa = \< f_1,\ldots,f_s\>\subseteq\KK[T_1,\ldots,T_r]$ be an ideal 
and let  $\gamma_0\seite\gamma$.
Set $g_i := (f_i)_{\gamma_0}$ and 
$G := (g_1,\ldots,g_s)$.

\begin{enumerate}
\item[\tiny 1] 
Use~\ref{rem:gradiator:algo} to compute a
matrix $P$ representing 
$\pi\colon \TT^r \to \TT^r_{\gamma_0}/H(\gamma_0)$
\item[\tiny 2] 
Apply~\ref{algo:pushforward} to $P$ to obtain
$\pi_*G:=(\pi_*g_1,\ldots,\pi_*g_s)$
\item[\tiny 3] 
{\bf if} $T_1\cdots T_n \in \sqrt{\langle \pi_*G\rangle}\subseteq \KK[T_1,\ldots,T_n]$
\item[\tiny 4] \myident {\bf return} \textsc{false}
\item[\tiny 5] {\bf return} \textsc{true}
\end{enumerate}
\end{algorithm}

\begin{proof}
The map $\pi$ is a good quotient for the 
$H(\gamma_0)$-action on $\TT^r_{\gamma_0}$.
Consequently,  we have
\[
 \pi\left(\bigcap_{i=1}^s V\left(\TT^{r}_{\gamma_0};\ g_i\right)\right) 
= \bigcap_{i=1}^s \pi\left(V\left(\TT^{r}_{\gamma_0};\,g_i\right)\right)
=  V\left(\TT^{n};\  \pi_*g_1,\ldots,\pi_*g_s\right)
\,
\]
by standard properties of good quotients~\cite[p.~96]{kraft:geometrische}.
This shows that $V(\TT^r_{\gamma_0};\ \aa_{\gamma_0})\not=\emptyset$ 
if and only if $V(\TT^n;\ \pi_*G))\not=\emptyset$.
\end{proof}

\begin{remark}
If the total number of terms occurring among the generators is 
low compared to the number of variables in the sense 
that $P=P_G$ in the first line of Algorithm~\ref{algo:afacespushed},
then we might speed up the algorithm
using linear algebra as follows.
Each term $\pi_*g_i$ is linear by construction.
Solve the linear system of equations $\pi_*G = 0$.
Then $\gamma_0$ is an $\aa$-face if and only if 
there is a solution in $\TT^n$.
\end{remark}

Let us briefly recall the connection to tropical geometry, 
compare e.g. \cite{jensen:tropvars}.
Given a monomial-free ideal $\aa\subseteq \KK[T_1,\ldots,T_r]$, 
its \textit{tropical variety} is
\[
\trop{\aa}
\ := \
\bigcap_{f\in\aa} \trop{f}
\ \subseteq \
\QQ^r,
\]
where $\trop{f}$ is the support of the codimension one 
skeleton of the normal fan of the Newton polytope of $f$.
By \cite{tevelev}, 
\begin{align}
 \gamma_0\seite\gamma\text{ is an \aa-face }
\qquad
\gdw
\qquad
\trop{\aa} \cap(\gamma_0^*)^\circ\not=\emptyset\,.
\label{eq:tevelev}
\end{align}

Fixing a fan structure on $\trop{\aa}$, this can 
be turned into a computable criterion.
Note however that $\trop{\aa}$ usually carries 
more information than needed to determine the
$\aa$-faces and is in general harder to compute 
(see \cite{jensen:tropvars} for an algorithm).

\begin{remark}\label{rem:afacesimprove}
To compute all $\aa$-faces, the number of calls to 
Algorithm \ref{algo:afacespushed} can be reduced 
by any of the following ideas.
\begin{enumerate}
 \item 
The \textit{tropical prevariety} of a generating set $(f_1,\ldots,f_s)$ of $\aa$ is the coarsest common refinement $\sqcap_i \Upsilon_i$ where $\Upsilon_i$ is the one-codimensional skeleton of the normal fan of the Newton polytope of $f_i$.
Then each face $\gamma_0\seite\gamma$ whose dual face $\gamma_0^*$ does not satisfy equation (\ref{eq:tevelev}) w.r.t $\sqcap_i \Upsilon_i$
is not an $\aa$-face.
 \item 
A face $\gamma_0\seite\gamma$ is not an $\aa$-face if and only if there is $f\in\aa$ such that exactly one vertex of the newton polytope of $f$  lies in $\gamma_0$; also compare \cite[Prop. 9.3]{cc1}.
Choosing any subset of $\aa$, we may identify some faces $\gamma_0\seite\gamma$ that are no $\aa$-faces.
 \item \textit{Veronese embedding}: 
Let $\gamma_0\seite\gamma$ be such that there are (classically) homogeneous generators $g_1,\ldots,g_s$ of $\aa_{\gamma_0}$ of degree $d\in \ZZ_{\geq 0}$.
The images of the $g_i$ under
\[
 \KK[T_{\gamma_0}] \to \KK[S_\mu;\ \mu_1+\ldots +\mu_r = d]\ ,\qquad
 T^\mu \mapsto S_\mu
\]
give a linear system of equations with coefficient matrix $A$.
If a Gauss-Jordan normal form of $A$ contains a row with exactly one non-zero entry, $\gamma_0$ is no $\aa$-face. 
Adding redundant generators to $\aa_{\gamma_0}$ refines this procedure.
\item 
Let $\sigma\in S_r$ be a permutation of (the indices of) the variables $T_1,\ldots,T_r$
that keeps the set of generators of $\aa$ invariant.
Then 
\[
\gamma_0\seite\gamma \text{ $\aa$-face} 
\qquad\gdw\qquad
\cone{e_{\sigma(i)};\ e_i\in\gamma_0} \text{ $\aa$-face }.
\]
Some of those permutations can be computed by assigning a both edge- and vertex-colored graph to the generators of $\aa$ and computing its automorphism group, e.g. using \cite{nauty}. 
\end{enumerate}
\end{remark}

\begin{remark}
The efficiency of Algorithm \ref{algo:afacespushed} 
depends on the algorithms used for both 
Gr\"obner bases and Smith normal forms.
An implementation using the respective built 
in functions of \texttt{Maple} gave the following 
timings. 

\begin{center}
 \footnotesize
\begin{tabular}{cccc}
    \thickhline & remark \ref{rem:afacesnaive} & algorithm  \ref{algo:afacespushed} with \ref{rem:afacesimprove}(2)\\ \hline
    \aa-faces of $\aa_{2,5}$ & $21$\,{\rm s} & $10$\,{\rm s}  \\
    \aa-faces of $\aa_{2,6}$ & $16$\,{\rm min} & $76$\,{\rm s}  \\ 
    \aa-faces of $\aa_{2,7}$ & $>\,3$ days  &  $24.8$\,{\rm h}  \\ 
    \aa-faces of $\aa_{2,3,3}$ & $4.03$\,{\rm h} &  $44.1$\,{\rm min}  \\ \thickhline
\end{tabular}
\end{center}

\noindent
There, $\aa_{2,n}$ stands for the respective 
Pl\"ucker ideal
and $\aa_{2,3,3}$ denotes the defining ideal of the 
Cox ring of the space $X(2,3,3)$ of complete rank two 
collineations~\cite[Thm. 1]{kollin}.
\end{remark}

%%%%%%%%%%%%%%%%% Examples %%%%%%%%%%5

\section{Examples}\label{section:examples}

We consider torus actions on the affine cone over 
the Grassmannian $G(2,n)$ induced by a diagonal 
action on the Pl\"ucker coordinate space $\KK^r$, 
where $r=\binom{n}{2}$.
Such actions will be encoded by assigning the variable 
$T_i$ the $i$-th column of a matrix $Q$.
Moreover, we write $\aa_{2,n}\subseteq \KK[T_1,\ldots,T_{r}]$ 
for the Pl\"ucker ideal.

We compute both, the GIT-fan of the torus action on
$V(\KK^r;\,\aa_{2,n})$ as well as the GIT fan of the ambient 
space $\KK^r$.
The latter coincides with the so-called 
\textit{Gelfand Kapranov Zelevinsky decomposition} 
$\gkz(Q)$, i.e. the coarsest common refinement of 
all normal fans having their rays among the
cones over the columns of $Q$.
In general, the Gelfand Kapranov Zelevinsky decomposition
is a refinement of the GIT-fan.
See \cite{coxlittleschenck} for a toric background.

Below, the drawings show (projections of) the intersections 
of the respective fans with the standard simplex.

%%%%%%%%%%%%%%%%%%%%%%%%%%%%%%%%%%%%%%%%%%%%%%%%%%%

\begin{example}
 For $n=4$, the ideal 
$\aa_{2,4} = \< T_1T_6-T_2T_5+T_3T_4\>\subseteq\KK[T_1,\ldots,T_6]$ 
is homogeneous with respect to 
\begin{center}
\begin{minipage}{.4\textwidth}
 \begin{align*}
Q\ :=\ 
\begin{bmatrix}
1 & 0 & 0 & 1 & 1 & 0\\
0 & 1 & 0 & 1 & 0 & 1\\
0 & 0 & 1 & 0 & 1 & 1
\end{bmatrix}\,.
\end{align*}
\end{minipage}
\quad
\begin{minipage}{.2\textwidth}
\begin{center}
\includegraphics{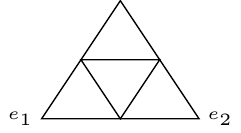}
\tiny
$\Lambda(\aa_{2,4},Q)$
\end{center}
\end{minipage}
\quad
\begin{minipage}{.2\textwidth}
\begin{center}
\includegraphics{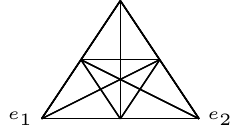}
\tiny
${\rm GKZ(Q)}$
\end{center}
\end{minipage}
\end{center}

Using Algorithm \ref{algo:gitfantraversal}, we obtain the four maximal 
GIT-chambers of $\Lambda(\aa_{2,4},Q)$.
The finer fan $\gkz(Q)$ has twelve maximal cones.
\end{example}

%%%%%%%%%%%%%%%%%%%%%%%%%%%%%%%%%%%%%%%%%%%%%%%%%%%

 \begin{example}
 For $n=5$, the ideal $\aa_{2,5} \subseteq\KK[T_1,\ldots,T_{10}]$ is homogeneous with respect to
\vspace*{-.3cm}
\begin{center}
\begin{minipage}{5cm}
 \begin{align*}
Q\ =\ 
\begin{bmatrix}
1 & 0 & 0 & 0 & 1 & 1 & 1 & 0 & 0 & 0\\
0 & 1 & 0 & 0 & 1 & 0 & 0 & 1 & 1 & 0\\
0 & 0 & 1 & 0 & 0 & 1 & 0 & 1 & 0 & 1\\
0 & 0 & 0 & 1 & 0 & 0 & 1 & 0 & 1 & 1
\end{bmatrix}\,.
\end{align*}
\end{minipage}
\hspace*{-.4cm}
\begin{minipage}{3cm}
\begin{center}
\includegraphics{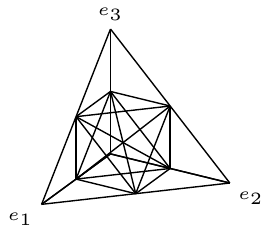}
\tiny
$\Lambda(\aa_{2,5},Q)$
\end{center}
\end{minipage}
\hspace*{-.4cm}
\begin{minipage}{3cm}
\begin{center}
\includegraphics{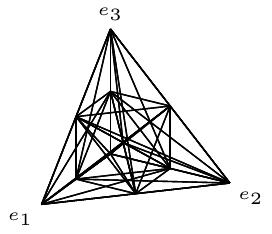}
\tiny
${\rm GKZ}(Q)$
\end{center}
\end{minipage}
\end{center}

By Algorithm \ref{algo:gitfantraversal}, there are twelve four-dimensional cones in $\Lambda(\aa_{2,5}, Q)$
whereas $\gkz(Q)$ contains $336$ such cones.
\end{example}

%%%%%%%%%%%%%%%%%%%%%%%%%%%%%%%%%%%%%%%%%%%%%%%%%%%

\begin{example}
 For $n=6$, the ideal $\aa_{2,6}\subseteq \KK[T_1,\ldots,T_{15}]$ is homogeneous with respect to 

\begin{center}
 \begin{minipage}{7cm}
\begin{align*}
Q\,=\,
\begin{bmatrix}
1 & 0 & 0 & 0 & 0 & 1 & 1 & 1 & 1 & 0 & 0 & 0 & 0 & 0 & 0\\
0 & 1 & 0 & 0 & 0 & 1 & 0 & 0 & 0 & 1 & 1 & 1 & 0 & 0 & 0\\
0 & 0 & 1 & 0 & 0 & 0 & 1 & 0 & 0 & 1 & 0 & 0 & 1 & 1 & 0\\
0 & 0 & 0 & 1 & 0 & 0 & 0 & 1 & 0 & 0 & 1 & 0 & 1 & 0 & 1\\
0 & 0 & 0 & 0 & 1 & 0 & 0 & 0 & 1 & 0 & 0 & 1 & 0 & 1 & 1
\end{bmatrix}\,.
\end{align*}
\end{minipage}
\ 
\end{center}
Using Algorithm \ref{algo:gitfantraversal}, we obtain the $81$ five-dimensional cones of $\Lambda(\aa_{2,6},Q)$. 
The fan $\gkz(Q)$ has $61920$ such cones.
\end{example}

%%%%%%%%%%%%%%%%%%%%%%%%%%%%%%%%%%%%%5


\begin{thebibliography}{0}
\bibitem{polyhedraldivisors}
K.~Altmann, J.~Hausen: 
Polyhedral Divisors and Algebraic Torus Actions.
{\it Mathematische Annalen} {\bf 334}(3) (2006), 557--607.
%
\bibitem{coxrings}
I.~Arzhantsev, U.~Derenthal, J.~Hausen, A.~Laface:
{\it Cox rings}. Preprint, \texttt{arXiv:1003.4229}; see also the 
authors' webpages.
%
\bibitem{gitviacoxrings}
I.~Arzhantsev, J.~Hausen:
Geometric invariant theory via Cox rings.
{\it Journal of Pure and Applied Algebra} {\bf 213}(1) (2009), 154--172.
% 
\bibitem{reversesearch}
D.~Avis, K.~Fukuda:
Reverse Search for Enumeration.
{\it Discr. Appl. Mathematics} {\bf 65}(1-3) (1993), 
21--46.
%
\bibitem{cc1}
F.~Berchtold, J.~Hausen:
Cox Rings and Combinatorics.
{\it Trans. Am. Math. Soc.} {\bf 359}(3) (2007), 1205--1252.
%
\bibitem{hausen:amplecone}
F.~Berchtold, J.~Hausen:
GIT-Equivalence beyond the ample cone.
{\it Michigan Math. J.} {\bf 54}(3) (2006), 483--515
%
\bibitem{jensen:tropvars} 
T.~Bogart, A.~Jensen, D.~Speyer, B.~Sturmfels, R.~Thomas:
Computing tropical varieties.
{\it Journal of Symbolic Computation} {\bf 42}(1-2) (2007), 54--73.
%
\bibitem{coxlittleschenck}
D.~Cox, J.~Little, H.~Schenck:
{\it Toric Varieties}.
(Graduate Studies in Mathematics.
American Mathematical Society, 2011).
% 
\bibitem{dolgachevhu:GIT}
I.~Dolgachev, Y.~Hu:
Variation of geometric invariant theory quotients.
{\it Pub. Mathematiques} {\bf 87}(1) (1998), 5--51
%
\bibitem{convex}
M.~Franz:
{\it Convex -- a Maple package for convex geometry}.
(Available at \texttt{http://www.math.uwo.ca/~mfranz/convex/})
%
\bibitem{groebfan}
K.~Fukuda, A.~Jensen, R.~Thomas:
Computing Gr\"obner fans.
{\it Math. Comput.} {\bf 76}(260) (2007),
2189--2212.
%
\bibitem{fulton}
W.~Fulton:
{\it Introduction to toric varieties}.
(2nd corrected printing.
Annals of mathematics studies 131,
Princeton University Press, 1997).
%
\bibitem{kollin}
J.~Hausen, M.~Liebend\"orfer:
{\it The Cox ring of the space of complete rank two collineations}.
Preprint, \texttt{arXiv:1110.1171}.
%
\bibitem{mdsgit}
Yi~Hu, S.~Keel:
Mori dream spaces and GIT. 
Dedicated to William Fulton on the occasion of his 60th birthday. 
{\it Michigan Math. J.} {\bf 48}(1)  (2000), 331--348. 
%
\bibitem{symmfans}
 A.~Jensen:
 Traversing Symmetric Polyhedral Fans.
 {\it ICMS'10 Proc. 3rd Intl. Congr. on Mathematical software}, (Springer, 2010), 282--294.
%
\bibitem{gitfanlib}
S.~Keicher:
{\it gitfanlib -- a package for GIT-fans}.
(Available at \texttt{http://www.mathematik.uni-tuebingen.de/\textasciitilde keicher/gitfanlib/}).
%
\bibitem{kraft:geometrische}
H.~Kraft: 
\textit{Geometrische Methoden in der Invariantentheorie}.
(Aspekte der Mathematik. Vieweg, 1985).
%
\bibitem{pointloc}
D.~Liu:
A note on point location in arrangements of hyperplanes.
{\it Inf. Process. Lett.} {\bf 90}(2) (2004), 93--95.
%
\bibitem{nauty}
B.~McKay:
{\it Nauty -- a program for computing automorphism groups of graphs and digraphs}.
(Available at 
\texttt{http://cs.anu.edu.au/~bdm/nauty/}).
%
\bibitem{mumford:GIT}
D.~Mumford, J.~Fogarty, F.~Kirwan:
\textit{Geometric invariant theory}. 
(Third edition. Ergebnisse der Mathematik und ihrer Grenzgebiete (2),
34. Springer-Verlag, Berlin, 1994).
%
\bibitem{tevelev}
J.~Tevelev:
Compactifications of subvarieties of tori.  
{\it Amer. J. Math.} {\bf 129}(4)  (2007), 1087--1104. 
%
\bibitem{thaddeus:GIT}
M.~Thaddeus:
Geometric Invariant Theory and Flips.
{\it J. Amer. Math. Soc.} {\bf 9}(3)  (1996),  691--723. 

\end{thebibliography}
\end{document}